\documentclass[12pt,letterpaper]{amsart}

\usepackage[none]{hyphenat}
\usepackage[utf8]{inputenc}
\usepackage[T1]{fontenc}
\usepackage{amsfonts}
\usepackage[leqno]{amsmath}
\usepackage{amssymb, comment, float}
\usepackage[dvipsnames]{xcolor}
\usepackage[foot]{amsaddr}
\usepackage[shortlabels]{enumitem}
\usepackage{mathdots}
\usepackage{todonotes}
\usepackage[footskip=0.5in, headheight = 0.5in, top=1.25in, bottom=1.25in, right=1in, left=1in]{geometry}
\usepackage{hyperref}
\usepackage[center]{caption}
\usepackage{eufrak}
\usepackage{marvosym}
\usepackage{latexsym}
\usepackage{tikz}
\usepackage{subfig}
\usepackage{thm-restate}

\hypersetup{
colorlinks,
linkcolor=blue,
urlcolor=black,
citecolor=blue,}

\newtheorem{theorem}{Theorem}[section]
\newtheorem{lemma}[theorem]{Lemma}

\newtheorem{observation}[theorem]{Observation}

\def\dd{\hbox{-}}

\newcommand{\poi}{\mathbb N}

\newcounter{tbox}
\newcommand{\sta}[1]{\medskip\medskip\refstepcounter{tbox}\noindent{\parbox{\textwidth}{(\thetbox) {\em #1}}}\vspace*{0.3cm}}
\newcommand{\mylongtitle}[1]{%
	\ifodd\value{page}%
		\protect\parbox{0.97\linewidth}{#1}\hfill%
	\else%
		\hfill\protect\parbox{0.97\linewidth}{#1}%
	\fi%
}

\title[Finding non-bisimplicial and triad-free minimal separators]{Finding non-bisimplicial and triad-free minimal separators}

\author{Maria Chudnovsky$^\dagger$}
\author{Ilya Maier$^{\ddagger}$}

\thanks{$^\dagger$ Princeton University, Princeton, NJ, USA.
	Supported by NSF Grant DMS-2348219, AFOSR grant FA9550-25-1-0275 and a Guggenheim Fellowship}
\thanks{$^{\ddagger}$ Princeton University, Princeton, NJ, USA.
	Supported by a scholarship of the German Academic Exchange Service (DAAD)}

\date{\today}
\begin{document}
\raggedbottom
\maketitle

\begin{abstract}

	A {\em $u\dd v$ separator} in $G$ is a set $S\subseteq V(G)\setminus\{u,v\}$ such that there is no $u\dd v$ path in $G\setminus S$.
	A {\em minimal separator} in $G$ is a set $S$ such that for some $u,v\in V(G)\setminus S$, $S$ is a $u\dd v$ separator in $G$ but no proper subset of $S$ is.
	A set $S$ is {\em bisimplicial} if it is the union of at most two cliques.
	Recognizing graphs in which every minimal separator is the union of at most $k$ cliques is polynomial-time solvable for $k=1$ and NP-hard for every fixed $k\geq 3$.
	We settle the remaining case $k=2$ by giving a polynomial-time algorithm to decide whether every minimal separator in a graph is bisimplicial.
	
	A set $S$ is {\em triad-free} if its independence number $\alpha(S)$ is at most two.
	Given a graph $G$ and two vertices $u,v$, the problem of finding a $u\dd v$ separator $S$ with $\alpha(S)\leq k$ had the same complexity gap: it was known to be polynomial-time solvable for $k=1$ and NP-hard for every fixed $k\geq 3$.
	We resolve the remaining case $k=2$ by giving a polynomial-time algorithm that finds a triad-free $u\dd v$ separator or determines that none exists.

\end{abstract}

\section{Introduction}
\label{sec:introduction}

The set of all positive integers is denoted by $\poi$.
Graphs in this paper have finite vertex sets, no loops, and no parallel edges.
Let $G$ and $H$ be two graphs.
$G$ is said to \textit{contain} $H$ if $G$ contains an induced subgraph isomorphic to $H$.
$G$ is \textit{$H$-free} if $G$ does not contain $H$.
For $X \subseteq V(G)$, we denote by $G[X]$ the subgraph of $G$ induced by $X$, and we refer to $X$ and $G[X]$ interchangeably.
By $\alpha(X)$, we denote the size of the largest stable set in $G[X]$.
For disjoint subsets $X, Y \subseteq V(G)$, we say that $X$ and $Y$ are \textit{anticomplete} (or that $X$ {\em is anticomplete to} $Y$) if there is no edge in $G$ with an end in $X$ and an end in $Y$.
Similarly, we say that $X$ and $Y$ are \textit{complete} (or that $X$ {\em is complete to} $Y$) if $xy\in E(G)$ for every pair $x\in X,y\in Y$.

An \textit{induced minor} of $G$ is a graph obtained from an induced subgraph of $G$ by repeatedly contracting edges and deleting all arising loops and parallel edges.
For a graph $H$, we say that $G$ {contains \em $H$ as an induced minor} if there is an induced minor of $G$ isomorphic to $H$.
We say that $G$ is {\em $H$-induced-minor-free} if $G$ does not contain $H$ as an induced minor.
For standard graph-theoretic terminology, the reader is referred to \cite{Diestel}.

Let $G$ be a graph.
We use $G^C$ to denote the {\em complement} of $G$.
A {\em path} in $G$ is a set of vertices $P\subseteq V(G)$ such that $G[P]$ is a path.
An {\em antipath} in $G$ is a set of vertices $A\subseteq V(G)$ such that $G^C[A]$ is a path.
An {\em antihole} in $G$ is a set of vertices $A\subseteq V(G)$ such that $G^C[A]$ is a cycle of length at least 4.
We say that an antipath or an antihole $A$ is {\em odd} if $|A|$ is odd.
Analogously, an antipath or an antihole $A$ is {\em even} if $|A|$ is even.

A {\em $u\dd v$ separator} in $G$ is a set $S\subseteq V(G)\setminus\{u,v\}$ such that there is no $u\dd v$ path in $G\setminus S$.
A {\em minimal separator} in $G$ is a set $S$ such that for some $u,v\in V(G)\setminus S$, $S$ is a $u\dd v$ separator in $G$ but no proper subset of $S$ is.
We say that a separator is {\em odd} if it contains an odd antihole.
A {\em shortest odd minimal separator} is an odd minimal separator $S$ that contains an odd antihole $A$ such that no other odd minimal separator $S'$ contains an odd antihole of length less than $|A|$.
A set $S$ is {\em bisimplicial} if $S$ is the union of at most two cliques.
Finally, a set $S$ is {\em triad-free} if its independence number $\alpha(S)$ is at most two.

Let $\mathcal G_k$ denote the class of all graphs in which every minimal separator is the union of at most $k$ cliques.
The class $\mathcal G_1$ corresponds to the class of all chordal graphs and can be recognized in polynomial time \cite{chordal, recognizing-chordal}.
It is known that the problem of recognizing graphs in $\mathcal G_k$ is NP-hard for $k\ge 3$ \cite{bisimplicial_separators}.
The problem of recognizing graphs in $\mathcal G_2$ was posed as an open problem in \cite{bisimplicial_separators}.
We show that it can be solved in polynomial time:

\begin{restatable}{theorem}{main}
\label{thm:main}

	There exists a polynomial-time algorithm to decide if a graph $G$ is in $\mathcal G_2$.

\end{restatable}

The proof outline of Theorem \ref{thm:main} is as follows.
As a first step, we test whether every minimal separator $S$ in $G$ is triad-free.
Using the following lemma, this is equivalent to testing whether $G$ is $K_{2,3}$-induced-minor-free.

\begin{lemma}[Dallard, Milanič, Štorgel \cite{separators}]
\label{lem:separator-alpha-2}
	
	Let $G$ be a graph.
	Then, every minimal separator $S$ in $G$ is triad-free if and only if $G$ is $K_{2,3}$-induced-minor-free.
	
\end{lemma}

The following theorem shows that this test can indeed be done in polynomial time.

\begin{theorem}[Dallard, Dumas, Hilaire, Milani\v{c}, Perez, Trotignon \cite{theta_induced_minor}]
\label{thm:k23-im}
	
	There exists a polynomial-time algorithm that decides if $G$ contains $K_{2,3}$ as an induced minor.
	
\end{theorem}

Let $\mathcal G_2'$ be the class of all graphs $G$ in which every minimal separator $S$ is triad-free but such that $G\notin \mathcal G_2$.
Note that this means that every graph $G\in \mathcal G_2'$ has an odd minimal separator.
Given a $K_{2,3}$-induced-minor-free graph $G$, we show how to find a shortest odd minimal separator in polynomial time, if one exists, which then implies Theorem \ref{thm:main}.

To do so, we find connected sets on opposite sides of a minimal separator whose neighborhoods contain a common odd antihole.
The structural lemmas allow us to find such sets by guessing a bounded number of vertices, even though the separator and the antihole may be large.

Our second result concerns the existence of separators between two prescribed vertices with bounded independence number.
For a fixed positive integer $k$, consider the problem of deciding whether a graph $G$ has a $u\dd v$ separator $S$ with $\alpha(S)\leq k$.
This problem is polynomial-time solvable for $k=1$, while Dallard, Fomin, Golovach, Korhonen, and Milanič \cite{alpha-separator-np-hard} proved that it is NP-complete for every fixed $k\geq 3$ and left the case $k=2$ open.
We settle this remaining case by giving a polynomial-time algorithm that finds a triad-free $u\dd v$ separator or determines that none exists.

\begin{restatable}{theorem}{triadfree}
\label{thm:triad-free-cutset}

	Let $u,v$ be two vertices in a graph $G$.
	There exists a polynomial-time algorithm that either finds a triad-free $u\dd v$ separator in $G$ or determines that none exists.
	
\end{restatable}

Note that if a triad-free $u\dd v$ separator exists, then there also exists a minimal triad-free $u\dd v$ separator.
Given a triad-free $u\dd v$ separator $S$, we can make it minimal by repeatedly removing a vertex of $S$ and testing if the remaining set is still a $u\dd v$ separator.

The proof proceeds as follows.
The main structural step behind our algorithm identifies, among polynomially many candidates, a subset of the desired separator $S$ whose removal forces the remainder of $S$ to be bisimplicial.
We then reduce the remaining task to the {\em list 2-clique cutset problem}.
An {\em instance $\mathcal I$ of the list 2-clique cutset problem} is a list $(U,V,P,Q)$ of subsets of $V(G)$ and a {\em solution} to $\mathcal I$ is a partition $(U',V',P',Q')$ of $V(G)$ such that
\begin{itemize}
	\item $U'\subseteq U, V' \subseteq V, P'\subseteq P,Q'\subseteq Q$,
	\item $U'$ is anticomplete to $V'$, and
	\item $P'$ and $Q'$ are both (possibly empty) cliques.
\end{itemize}
It is known that the list 2-clique cutset problem can be solved in polynomial time \cite{list-two-clique-cutset}.

\begin{theorem}[Cameron, Eschen, Ho{\`a}ng and Sritharan \cite{list-two-clique-cutset}]
\label{thm:list-2-clique-cutset}

	The list 2-clique cutset problem admits a polynomial-time algorithm.
	
\end{theorem}

As a corollary of Theorem \ref{thm:triad-free-cutset}, we obtain a polynomial-time algorithm for the {\em list triad-free cutset problem}.
An {\em instance $\mathcal I$ of the list triad-free cutset problem} is a list $(U,V,S)$ of subsets of $V(G)$ and a {\em solution} to $\mathcal I$ is a partition $(U',V',S')$ of $V(G)$ such that
\begin{itemize}
	\item $U'\subseteq U, V' \subseteq V, S'\subseteq S$,
	\item $U'$ is anticomplete to $V'$, and
	\item $S'$ is triad-free.
\end{itemize}
We think of $U'$ and $V'$ as the two sides of the separation and we think of the list $(U,V,S)$ as the list of allowed vertices in each corresponding set.
Thus, we may prescribe vertices on either side of the separation, as well as vertices that must belong to, or must be excluded from, the cutset.
The goal is to either find such a partition or to determine that none exists.

\begin{restatable}{corollary}{listtriadfree}
\label{cor:list-triad-free-cutset}

	The list triad-free cutset problem admits a polynomial-time algorithm.
	
\end{restatable}

We will explain how to derive Corollary \ref{cor:list-triad-free-cutset} in Section \ref{sec:corollaries}.

Additionally, we define two extensions of the list triad-free cutset problem.
The {\em connected list triad-free cutset problem} is defined the same way as the list triad-free cutset problem, except that for a solution $(U',V',S')$, the set $S'$ is additionally required to be either connected or empty.
The {\em minimal list triad-free cutset problem} is defined the same way as the list triad-free cutset problem, except that for a solution $(U',V',S')$, each vertex $s\in S'$ is additionally required to have a neighbor both in $U'$ and $V'$.
We show a sharp boundary in the hardness of these extensions:
While the connected version of the problem is still polynomial-time solvable, the minimal version is NP-hard.

\begin{restatable}{corollary}{connectedlisttriadfree}
\label{cor:conn-list-triad-free-cutset}
	
	The connected list triad-free cutset problem admits a polynomial-time algorithm.
	
\end{restatable}

\begin{restatable}{proposition}{minimallisttriadfree}
\label{pro:minimal-list-triad-free-cutset}
	
	The minimal list triad-free cutset problem is NP-hard, even when the input graph $G$ is the complement of a bipartite graph.
	
\end{restatable}

\subsection{Outline of the paper}

In Section \ref{sec:definitions}, we introduce some definitions.
Then, in Section \ref{sec:structure}, we show several structural properties of shortest odd minimal separators.
In Section \ref{sec:algorithm}, we use the structural properties to devise the algorithm from Theorem \ref{thm:main}.
In Section \ref{sec:triad-free}, we prove Theorem \ref{thm:triad-free-cutset}.
Finally, in Section \ref{sec:corollaries}, we derive Corollaries \ref{cor:list-triad-free-cutset} and \ref{cor:conn-list-triad-free-cutset}, and Proposition \ref{pro:minimal-list-triad-free-cutset}.

\section{Definitions}
\label{sec:definitions}

Let $G$ be a graph.
For $S\subseteq V(G)$, we define the \textit{neighborhood} of $S$, denoted $N(S)$, to be the set of all vertices in $V(G) \setminus S$ with at least one neighbor in $S$.
By the {\em closed neighborhood of $S$}, denoted $N[S]$, we mean the set $N(S) \cup S$.
We say that a set $X\subseteq V(G)\setminus S$ is {\em $S$-full} or {\em full to $S$} if $S\subseteq N(X)$.
Note that if $S$ is a minimal separator in $G$, then $G\setminus S$ has at least two $S$-full components \cite{algo_graph_theory}.
We define the {\em non-neighborhood} of $S$, denoted $M(S)$, to be the set $V(G) \setminus N[S]$.
Note that a vertex $v$ is not its own non-neighbor.
For a set $A\subseteq V(G)$, we define $N_A(S)$ to be the set $N(S)\cap A$.
Similarly, $M_A(S) = M(S)\cap A$.
For convenience, we will treat a vertex $x$ and the set $\{x\}$ interchangeably.
Analogously, we will treat a pair $uv$ and the set $\{u,v\}$ interchangeably.

Let $A$ be an antihole in $G$.
We write $A=a_1\dd \dots\dd a_l$ to mean that each consecutive pair of vertices $a_i,a_{i+1}$ is non-adjacent.
For convenience, we identify $a_i$ with $a_{i+kl}$ and $a_{i-kl}$ for every $k\in \poi$.
Let $P=p_1\dd \dots \dd p_k$ be a path or an antipath in $G$.
We call the vertices $p_1$ and $p_k$ the \textit{ends} of $P$ and refer to $P \backslash\left\{p_1, p_k\right\}$ as the \textit{interior} of $P$, denoted $P^*$.

Let $S\subseteq V(G)$ be a set of vertices and $v\in V(G)$.
We say that a vertex $v$ is {\em small for $S$} if $v$ has at most 2 non-neighbors in $S$.
Similarly, $v$ is {\em big for $S$} if $v$ has at least 3 non-neighbors in $S$.
The set of vertices of $V(G)\setminus S$ that are big for $S$ is denoted by $B(S)$.
When the set $S$ is clear, we just say that $v$ is {\em small} or {\em big}.
For an antihole $A=a_1\dd \dots\dd a_l$, we define a {\em stretch of $v$} to be an antipath $P\subset A$ with $|P|\ge 2$ such that $v$ is not adjacent to the ends of $P$ and is complete to $P^*$.
We say that a stretch $P$ is {\em even} if $P$ is even, and analogously a stretch $P$ is {\em odd} if $P$ is odd.
A stretch $P$ is {\em long} if $|P|\ge 5$.

\section{The structure of $\mathcal G_2'$}
\label{sec:structure}

Throughout this section, let $G\in \mathcal G_2'$ be a graph, let $S$ be a shortest odd minimal separator in $G$, and let $A=a_1\dd \dots\dd a_l$ be a shortest odd antihole in $S$.

\begin{lemma}
\label{lem:has-even-stretch}

	Every $v\in B(A)$ has an even stretch in $A$.
	
\end{lemma}

\begin{proof}

	Let $M(v)\cap A = \{a_{i_1},\dots,a_{i_k}\}$ where $i_1< i_2 < \dots < i_k$ and note that $k \ge 3$.
	For $1\le j<k$, let $A_j=a_{i_j}\dd a_{i_j+1}\dd \dots \dd a_{i_{j+1}}$ and let $A_k=a_{i_k}\dd a_{i_k+1}\dd \dots \dd a_{i_1}$.
	Note that each $A_j$ is a stretch of $v$.
	Observe that $A=\cup_{i=1}^k A_i$ and that $\sum_{i=1}^k |A_i| = |A|+k$ because every vertex in $M(v)\cap A$ is counted twice.
	Since $|A|$ is odd, it follows that one of the stretches is even.
	This completes the proof.
	
\end{proof}

\begin{lemma}
\label{lem:every-s-is-small}
	
	Every vertex $s\in S$ is small for $A$.
	
\end{lemma}

\begin{proof}
	
	Suppose that $s\in S$ is big.
	Then, by Lemma \ref{lem:has-even-stretch}, $s$ has an even stretch $P$ in $A$.
	Since $\alpha(S) = 2$, $s$ does not have two consecutive non-neighbors in $A$.
	Together with the fact that $s$ is big, it follows that $2 < |P| \le |A|-3$.
	Let $A' = P \cup \{s\}$.
	Then, $A'$ is an odd antihole in $S$ shorter than $A$, which is a contradiction.
	This completes the proof.
	
\end{proof}

\begin{lemma}
\label{lem:at-most-6-small}
	
	Let $L$ be an $S$-full component of $G\setminus S$.
	Then, there exists a subset $L'\subseteq L$ such that $G[L']$ is connected, $L'$ is full to $A$, and such that $L'$ contains at most 6 small vertices.
	
\end{lemma}

\begin{proof}

	Let $L' \subseteq L$ be connected and full to $A$, such that $|L'|$ is as small as possible.
	Suppose first that $L'$ is a path.
	Let $l_1,l_2$ be the ends of $L'$.
	By the minimality of $|L'|$, for $i \in \{1,2\}$ there exists $a(i) \in A$ such that $l_i$ is the unique neighbor of $a(i)$ in $L'$.
	It follows that for every vertex $v$ in $(L')^*$ that is small for $A$, $M_A(v)\subseteq \{a(1),a(2)\}$.
	Since $A$ is odd, there exist two consecutive vertices $a_k,a_{k+1}$ that are complete to all the small vertices in $(L')^*$.
	Suppose that $(L')^*$ contains at least 5 small vertices.
	Then, 3 of them form a stable set which together with $\{a_k,a_{k+1}\}$ forms a $K_{2,3}$ in $G$, contrary to the fact that $G\in \mathcal G_2'$.
	Thus, $L'$ contains at most 6 small vertices as required.

	Now, we may assume that $L'$ is not a path, and so there exist three distinct vertices $l_1,l_2,l_3 \in L'$ such that for each $i$, $L' \setminus l_i$ is connected.
	By the minimality of $|L'|$, for each $i$, there is $a(i) \in A$ such that $l_i$ is the unique neighbor of $a(i)$ in $L'$.
	It follows that the vertices $a(1)$, $a(2)$ and $a(3)$ are all distinct and every vertex of $L' \setminus \{l_1,l_2,l_3\}$ is anticomplete to $\{a(1), a(2),a(3)\}$.
	Consequently, no vertex of $L' \setminus \{l_1,l_2,l_3\}$ is small, and the assertion of the lemma holds.

\end{proof}

If $A$ were known, we could enumerate sets of at most six small vertices and take components of the graph induced by each such set together with $B(A)$.
By Lemmas \ref{lem:every-s-is-small} and \ref{lem:at-most-6-small}, one of these guesses yields a component that is full to $A$ and lies entirely in an $S$-full component of $G\setminus S$.
The next two lemmas allow us to search through polynomially many candidates for either a shortest odd antihole of $S$ or a subset of $A$ of size at most 22 that preserves all big vertices in an $S$-full component.

\begin{lemma}
\label{lem:long-stretch}

	Suppose that $b\in B(A)$ has a long stretch $P=\{a_i,\dots,a_j\}$ in $A$ where $i<j$.
	Then, every other stretch of $b$ has exactly 3 vertices.

\end{lemma}

\begin{proof}
	
	First, suppose that $b$ is anticomplete to two consecutive vertices $a_k,a_{k+1}$.
	Note that in that case $X=\{b,a_k,a_{k+1}\}$ is a stable set of size 3.
	Since $b$ is big, $|X\cap P|\le 1$.
	Hence, there exist $a_p,a_{p+1}\in P\setminus X$ that are complete to $X$.
	This contradicts $G$ being $K_{2,3}$-free.

	Thus, $b$ is adjacent to at least one of $a_k,a_{k+1}$ for every $k \in \poi$.
	By Lemma \ref{lem:has-even-stretch}, $b$ has an even stretch in $A$.
	Let $E=\{a_q,\dots,a_w\}$ be a shortest even stretch of $b$ in $A$ and let $A' = E\cup \{b\}$.
	Suppose that $b$ has a long stretch $P'$ in $A\setminus E^*$.
	Since $b$ is big for $A$, $A'$ is shorter than $A$ and $|E\cap P'|\le 1$.
	Hence, there exist $a_p,a_{p+1}\in P'\setminus E$ that are complete to $A'$.
	Then, there exists a minimal $a_p\dd a_{p+1}$-separator $S'$ in $G$ that contains $A'$, contrary to $S$ being a shortest odd minimal separator in $G$.
	It follows that $E=P$ and by minimality of $E$ every stretch of $b$ in $A\setminus P^*$ has size at most 3.
	Since $b$ has no two consecutive non-neighbors in $A$, every other stretch has size exactly three.
	This completes the proof.
	
\end{proof}

\begin{lemma}
\label{lem:cover-big-vertices}
	
	Let $B\subseteq B(A)$.
	Then, one of the following holds:
	\begin{enumerate}
		\item There exist two vertices $b_1,b_2\in B$ such that $A\subseteq N(\{b_1,b_2\})$.
		\item There exists a set $A'\subseteq A$ of size at most 22 such that $B\subseteq B(A')$.
	\end{enumerate}
	
\end{lemma}

\begin{proof}

	We may assume $B$ is non-empty.
	Let $b \in B$ be a vertex with the longest stretch in $A$, say $\{a_i,\ldots,a_j\}$ where $i<j$.
	Suppose that some vertex $b' \in B$ is complete to $\{a_{i-1},a_i,a_{i+1},a_{j-1},a_j,a_{j+1}\}$.
	By Lemma \ref{lem:long-stretch}, it follows that either $b'$ has a longer stretch than $b$, which is a contradiction, or outcome (1) holds.

	Therefore, we may assume that every vertex $b' \in B$ is non-adjacent to one of $a_{i-1},a_i,a_{i+1},\allowbreak a_{j-1},a_j,a_{j+1}$.
	Let $A'=\{a_{i-5},\ldots,a_{i+5},\allowbreak a_{j-5}, \ldots,a_{j+5}\}$.
	We show that every $b'\in B$ is big for $A'$.
	If $b'$ has a long stretch, then this follows from Lemma \ref{lem:long-stretch}.
	Otherwise, every stretch of $b'$ is of size at most $4$, and hence $b'$ is also big for $A'$.
	Thus, outcome (2) holds.
	This completes the proof.
	
\end{proof}

\section{Finding a non-bisimplicial minimal separator}
\label{sec:algorithm}

We will use the following two results from \cite{shortest_odd_hole} and \cite{separators}:

\begin{theorem}[Chudnovsky, Scott, Seymour \cite{shortest_odd_hole}]
\label{thm:find-shortest-odd}
	
	There exists a polynomial-time algorithm that finds a shortest odd hole in a graph $G$ or determines that none exists.
	
\end{theorem}

\begin{lemma}[Dallard, Milanič, Štorgel \cite{separators}]
\label{lem:subgraph-separator}
	
	Let $G$ be a graph, let $G'$ be an induced subgraph of $G$, and let $S'$ be a minimal separator in $G'$.
	Then there exists a minimal separator $S$ in $G$ such that $S' \subseteq S$.
	
\end{lemma}

We can now prove Theorem \ref{thm:main}.

\main*

\begin{proof}
	
	First, we describe the algorithm.
	
	Applying the algorithm from Theorem \ref{thm:k23-im}, we decide if $G$ is $K_{2,3}$-induced-minor-free.
	If not, we output ``$G\notin \mathcal G_2$''.
	Next, let $\mathcal A$ be the set of subsets of $V(G)$ of size at most 22.
	For each 4-tuple $(l_1,l_2,r_1,r_2)$ of vertices of $G$, add a shortest odd antihole $A$ of $N(\{l_1,l_2\})\cap N(\{r_1,r_2\})$, if one exists, to $\mathcal A$.
	This is possible by Theorem \ref{thm:find-shortest-odd}.
	For each $A\in \mathcal A$ and each set $L\subseteq V(G)\setminus A$ of size at most 6 do the following:
	\begin{enumerate}
		\item Mark each vertex of $V(G)\setminus A$ as big or small for $A$.
		\item For each component $C_1$ of $G[B(A)\cup L]$ and each component $C_2$ of $G\setminus N[C_1]$, test whether $N(C_1)\cap N(C_2)$ contains an odd antihole and if so, output ``$G\notin \mathcal G_2$''.
	\end{enumerate}
	If the algorithm has not yet output ``$G\notin \mathcal G_2$'', then output ``$G\in \mathcal G_2$''.
	
	\
	
	{\em Proof of correctness:}
	
	\
	
	It is clear that the algorithm runs in polynomial time.
	By Lemma \ref{lem:separator-alpha-2}, if $G\in \mathcal G_2$ then $G$ is $K_{2,3}$-induced-minor-free.
	Thus, if the algorithm detects a $K_{2,3}$ induced minor, the output ``$G\notin \mathcal G_2$'' is correct.
	Note also that every other time our algorithm outputs ``$G\notin \mathcal G_2$'', it finds an odd minimal separator in an induced subgraph of $G$.
	Lemma \ref{lem:subgraph-separator} implies that the output ``$G\notin \mathcal G_2$'' is always correct.
	Therefore, it suffices to show that if $G$ is $K_{2,3}$-induced-minor-free and $G\notin \mathcal G_2$, then our algorithm outputs ``$G\notin \mathcal G_2$''.
	
	Thus, suppose that $G$ is $K_{2,3}$-induced-minor-free but not in $\mathcal G_2$.
	Then, there exists a minimal separator $S$ in $G$ that is not bisimplicial.
	This means that $G[S]^C$ is not bipartite and therefore contains an odd cycle.
	Let $C$ be a shortest odd induced cycle in $G[S]^C$.
	As $\alpha(S)\le 2$, $G[S]^C$ is triangle-free and therefore $C$ is an odd hole.
	Thus, $G$ contains an odd minimal separator.
	Our goal now is to detect an odd minimal separator in $G$.
	
	Let $\tilde S$ be a shortest odd minimal separator in $G$, and let $\tilde L,\tilde R$ be two $\tilde S$-full components of $G\setminus \tilde S$.
	We claim that:
	
	\sta{\label{sta:core-for-big}
	There exists a set $A\in \mathcal A$ and a shortest odd antihole $\tilde A\supseteq A$ in $\tilde S$ such that either $\tilde L \cap B(\tilde A)\subseteq B(A)$ or $\tilde R \cap B(\tilde A)\subseteq B(A)$.}
	
	Suppose that there exist $l_1,l_2\in \tilde L$ and $r_1,r_2\in \tilde R$ such that $X=N(\{l_1,l_2\})\cap N(\{r_1,r_2\})$ contains a shortest odd antihole of $\tilde S$.
	Then, our algorithm will detect such a 4-tuple $(l_1,l_2,r_1,r_2)$ and will add a shortest odd antihole $A$ of $X$ to $\mathcal A$.
	Note that $A$ is then also a shortest odd antihole of $\tilde S$.
	Taking $\tilde A=A$, the claim follows.
	Otherwise, by possibly interchanging $\tilde L$ and $\tilde R$, we may assume that there exists a shortest odd antihole $\tilde A$ of $\tilde S$ such that $\tilde A\not \subseteq N(\{l_1,l_2\})$ for every pair $l_1,l_2\in \tilde L$.
	Then, Lemma \ref{lem:cover-big-vertices} implies that there exists a set $A\subseteq \tilde A$ of size at most 22, such that $\tilde L \cap B(\tilde A)\subseteq B(A)$.
	This completes the proof of \eqref{sta:core-for-big}.
	
	Let $A\in \mathcal A$ be a set and $\tilde A\supseteq A$ be a shortest odd antihole in $\tilde S$ as in \eqref{sta:core-for-big}.
	By possibly interchanging $\tilde L$ and $\tilde R$, we may assume that $\tilde L \cap B(\tilde A)\subseteq B(A)$.
	By Lemma \ref{lem:every-s-is-small}, every $s\in \tilde S$ is small for $\tilde A$ and hence also for $A$.
	Therefore, $B(A)$ is disjoint from $\tilde S$.
	Lemma \ref{lem:at-most-6-small} implies that in some iteration, our algorithm will find a connected set $C_1\subseteq \tilde L$ that is full to $\tilde A$.
	Hence, it will also find $C_2\supseteq \tilde R$ such that $\tilde A\subseteq N(C_1)\cap N(C_2)$ and will output ``$G\notin \mathcal G_2$'' as required.
	This completes the proof of correctness.
	
\end{proof}

\section{Finding a triad-free separator}
\label{sec:triad-free}

For a graph $G$, let $\overline E(G)=\{uv \mid u,v\in V(G) \text{ such that $u$ is not adjacent to $v$}\}$ be the set of {\em non-edges} of $G$.

\begin{observation}
\label{obs:dominating-non-edge}

	Let $S\subseteq V(G)$ be a set with $\alpha(S)=2$.
	Then, for every $e\in \overline E(S)$, $S \subseteq N[e]$.
	In other words, $M(e)$ is disjoint from $S$.
	
\end{observation}

In the following lemma, we show how to find a set $A$ that contains a neighbor of at least one endpoint of every non-edge in a minimal triad-free separator $S$.
Then, in Lemma \ref{lem:two-cliques}, we show how to use such sets $A$ to reduce the problem to finding a bisimplicial separator.

\begin{lemma}
\label{lem:dominating-subset}
	
	Let $S$ be a minimal separator in $G$ with $\alpha(S)=2$ and let $X$ be an $S$-full component of $G\setminus S$.
	Then, there exists a non-edge $e\in \overline E(S)$, a vertex $x\in X$, and a component $X'$ of $G[M(e)\cup \{x\}]$ such that $X'\subseteq X$ and $N(X') \cap e' \neq \emptyset$ for every $e'\in \overline E(S)$.
	
\end{lemma}

\begin{proof}
	
	For $e\in \overline E(S)$ and $x\in X$, let $C(e,x)$ denote the component of $G[M(e)\cup \{x\}]$ containing $x$.
	Choose $e\in \overline E(S)$ and $x\in X$ such that $C(e,x)$ is maximal, i.e., $C(e,x)\not \subset C(e',x')$ for all $e'\in \overline E(S),x'\in X$.
	Suppose that there exists $e'\in \overline E(S)$ such that $N(C(e,x))\cap e' = \emptyset$.
	Then, since $X$ is connected and full to $S$, there exists a vertex $x' \in X \cap N(C(e,x))$.
	But then $C(e,x)\subset C(e',x')$, which is a contradiction.
	Finally, Observation \ref{obs:dominating-non-edge} implies that $M(e)$ is disjoint from $S$, and hence $C(e,x) \subseteq X$ as required.
	
\end{proof}

\begin{lemma}
\label{lem:two-cliques}
	
	Let $S\subseteq V(G)$ be a set and let $X_1,X_2\subseteq V(G)\setminus S$ be such that for each $i\in\{1,2\}$ and every $e\in \overline E(S)$, $N(X_i)\cap e \neq \emptyset$.
	Then, $S\setminus(N(X_1)\cap N(X_2))$ is bisimplicial.
	
\end{lemma}

\begin{proof}
	
	Define \begin{itemize}
		\item $S' = N(X_1)\cap N(X_2)$,
		\item $P = (S \cap N(X_1)) \setminus N(X_2)$, and
		\item $Q = S \setminus (S'\cup P)$.
	\end{itemize}
	We claim that $P$ and $Q$ are both cliques.
	Suppose that $e\in \overline E(S)$ is contained in $P$.
	Then, $N(X_2)\cap e = \emptyset$, which is a contradiction.
	Otherwise, suppose that $e$ is contained in $Q$.
	Observe that $S\cap N(X_1) \subseteq S'\cup P$.
	Since $N(X_1) \cap e \neq \emptyset$ and $Q$ is disjoint from $S'\cup P$, the claim follows.
	It is straightforward to observe that $P\cup Q = S\setminus S'$.
	
\end{proof}

We are now ready to prove Theorem \ref{thm:triad-free-cutset}.
The intuition behind our algorithm is to first remove a subset $S'$ of the desired separator $S$ to make it bisimplicial, so we can reduce to the list 2-clique cutset problem.
We then create an auxiliary graph by adding some gadget vertices (the set $D$ in the algorithm description) that together with the lists of allowed vertices ensure that the two cliques found can be combined with $S'$ while still remaining triad-free.

\triadfree*

\begin{proof}
	
	First, we describe the algorithm.
	
	For $e\in \overline E(G)$ and $x\in V(G)$, let $C(e,x)$ denote the component of $G[M(e)\cup \{x\}]$ containing $x$.
	Let $\mathcal C = \{\emptyset\}\cup \{C(e,x)\mid e\in \overline E(G), x\in V(G)\}$.
	For every pair $C_1,C_2\in \mathcal C$, define $S=N(C_1)\cap N(C_2)$ and do the following.
	Let \[D = \{d_{xy}\mid xy\in \overline E(G\setminus S) \text{ such that } \alpha(S\cup \{x,y\})\ge 3 \}\] be a set of new vertices.
	Define a new graph $G'$ with vertex set $(V(G)\setminus S)\cup D$ and edges as follows:
	\begin{itemize}
		\item $G'[V(G)\setminus S]=G\setminus S$, and
		\item $N(d_{xy}) = V(G')\setminus \{x,y,d_{xy}\}$ for every $d_{xy}\in D$.
	\end{itemize}
	We create an instance $\mathcal I=(U,V,P,Q)$ of the list 2-clique cutset problem for $G'$ in the following way.
	Let $X=\{x\in V(G) \mid \alpha(S\cup \{x\})\ge 3\}$ and set
	\begin{itemize}
		\item $U = V(G')\setminus (\{v\} \cup D)$,
		\item $V = V(G')\setminus (\{u\} \cup D)$,
		\item $P = Q = V(G')\setminus (X\cup \{u,v\})$.
	\end{itemize}
	Suppose that $\mathcal I$ admits a solution $(U',V',P',Q')$.
	Then, set $S'= (S\cup P'\cup Q')\cap V(G)$.
	Check if $S'$ is a triad-free $u\dd v$ separator in $G$, and if so then output it.
	If after all iterations no triad-free $u\dd v$ separator in $G$ has been found, then output ``no solution exists''.
	
	\
	
	{\it Proof of correctness:}
	
	\
	
	Computing the set $\mathcal C$ takes polynomial time since there are at most $O(|V(G)|^3)$ pairs $(e,x)$ for which we need to compute $C(e,x)$.
	By Theorem \ref{thm:list-2-clique-cutset}, constructing an instance $\mathcal I$ and solving it also takes polynomial time.
	Lastly, checking that $S'$ is a triad-free $u\dd v$ separator can also be done in polynomial time.
	
	It is clear that if our algorithm outputs a set $S'$, then it is a triad-free $u\dd v$ separator.
	Hence, it suffices to show that if there exists a triad-free $u\dd v$ separator, then our algorithm will find it.
	
	Thus, suppose that there exists a triad-free $u\dd v$ separator.
	Note that then there exists a minimal triad-free $u\dd v$ separator $\tilde S$.
	Suppose that $\tilde S$ is a clique.
	Then, for $C_1=C_2=\emptyset$ our algorithm will set $S=D=\emptyset$, meaning $G'=G$.
	Thus, the created instance $\mathcal I$ of the list 2-clique cutset problem will admit a solution $(\tilde U,\tilde V,\tilde S,\emptyset)$ for some sets $\tilde U,\tilde V$ with $u\in \tilde U,v\in \tilde V$.
	By Theorem \ref{thm:list-2-clique-cutset}, the list 2-clique cutset algorithm will find a solution $(U',V',P',Q')$ to $\mathcal I$ and thus a bisimplicial $u\dd v$ separator $P'\cup Q'$.
	
	Therefore, we may assume that $\alpha(\tilde S)=2$.
	Let $\tilde U,\tilde V$ be the $\tilde S$-full components of $G\setminus \tilde S$ containing $u$ and $v$, respectively.
	Let $\tilde U'$ be the union of all components of $G\setminus (\tilde S\cup \tilde V)$.
	By Lemma \ref{lem:dominating-subset}, $\mathcal C$ contains sets $C_1\subseteq \tilde U$ and $C_2\subseteq \tilde V$ such that for each $C_i$ and every non-edge $e\in \overline E(\tilde S)$, $N(C_i)\cap e\neq \emptyset$.
	Let $S=N(C_1)\cap N(C_2)$.
	By Lemma \ref{lem:two-cliques}, we get that $\tilde S\setminus S$ is bisimplicial.
	Observe that $D$ is a clique in $G'$ such that each $d\in D$ has at most one non-neighbor in $\tilde S\setminus S$.
	Thus, the created instance $\mathcal I$ of the list 2-clique cutset problem admits a solution $(\tilde U',\tilde V, \tilde P, \tilde Q)$ for some cliques $\tilde P,\tilde Q$ such that $\tilde P \cup \tilde Q = (\tilde S\setminus S)\cup D$.
	Let $(U',V',P',Q')$ be a solution to $\mathcal I$ that our algorithm found.
	By the way we defined the instance, we get that $u\in U',v\in V'$ and that $U'$ is anticomplete to $V'$.
	It follows that $S'=(S\cup P'\cup Q')\cap V(G)$ is a $u\dd v$ separator in $G$.
	Now, suppose that $S'$ contains a triad $\{a,b,c\}$.
	Since $P'$ and $Q'$ are both cliques and are disjoint from $X=\{x\in V(G) \mid \alpha(S\cup \{x\})\ge 3\}$, it follows that (after possibly relabeling $a,b,c$) $a\in S,b\in P',c\in Q'$.
	Note that $D\subseteq P'\cup Q'$ follows by our definition of $\mathcal I$.
	But then $\{d_{bc},b,c\}$ is a triad in $P'\cup Q'$, which is a contradiction.
	It follows that $S'$ is a triad-free $u\dd v$ separator in $G$.
	This completes the proof.
	
\end{proof}

\section{Corollaries of Theorem \ref{thm:triad-free-cutset}}
\label{sec:corollaries}

\listtriadfree*

\begin{proof}
	
	We will show a reduction to the problem of finding a $u\dd v$ triad-free separator.
	Let $G$ be a graph and let $\mathcal I=(U,V,S)$ be an instance of the list triad-free cutset problem.
	Let $X=V(G)\setminus S$ and let $A=\{a^1_x,a^2_x\mid x\in X\}$.
	Obtain $G'$ from $G$ by adding the vertices $\{u,v\}\cup A$ such that
	\begin{itemize}
		\item $N(u)=V(G')\setminus (V\cup \{u,v\})$,
		\item $N(v)=V(G')\setminus (U\cup \{u,v\})$, and
		\item $N(a^1_x)=N(a^2_x)=V(G')\setminus \{x,a^1_x,a^2_x\}$ for every $x\in X$.
	\end{itemize}
	Use the algorithm from Theorem \ref{thm:triad-free-cutset} to find a triad-free $u\dd v$ separator $S'$ in $G'$.
	If none exists, output ``no solution exists''.
	Otherwise, let $U'$ be the component of $G'\setminus S'$ containing $u$, let $V'$ be the union of all the other components of $G'\setminus S'$, and output $(U'\cap V(G),V'\cap V(G),S'\cap V(G))$.
	
	\
	
	{\it Proof of correctness:}
	
	\
	
	We show that $\mathcal I$ admits a solution if and only if there exists a triad-free $u\dd v$ separator in $G'$.
	
	Suppose that $\mathcal I$ admits a solution $(\tilde U,\tilde V,\tilde S)$.
	Then, $\{u\}\cup \tilde U$ is anticomplete to $\{v\}\cup \tilde V$ in $G'$.
	Let $A$ be as defined by our algorithm and note that since $\tilde S\subseteq S$, we have that $\tilde S\cup A$ is triad-free.
	Thus, $\tilde S\cup A$ is a triad-free $u\dd v$ separator in $G'$.
	
	For the other direction, suppose that there exists a triad-free $u\dd v$ separator $S'$ in $G'$.
	Define $U',V'$ as in the algorithm description.
	It is straightforward to check that $(U'\cap V(G),V'\cap V(G),S'\cap V(G))$ is a partition of $V(G)$.
	Since $U'$ and $V'$ consist of disjoint components of $G'\setminus S'$, we have that $U'\cap V(G)$ is anticomplete to $V'\cap V(G)$.
	Observe that $A\subseteq S'$ since both $u$ and $v$ are complete to it.
	It follows that $S'$ cannot contain a vertex $x$ from $X=V(G)\setminus S$ because then $\{a^1_x,a^2_x,x\}$ would form a triad, which is a contradiction.
	Hence, from the definitions of $N(u)$ and $N(v)$, it follows that $U'\cap V(G)\subseteq U,V'\cap V(G)\subseteq V$ and $S'\cap V(G)\subseteq S$.
	Lastly, $S'\cap V(G)$ is triad-free since $S'$ is triad-free.
	Therefore, $(U'\cap V(G),V'\cap V(G),S'\cap V(G))$ is a solution to $\mathcal I$.
	This completes the proof.
	
\end{proof}

\connectedlisttriadfree*

\begin{proof}
	
	Let $G$ be a graph and let $\mathcal I=(U,V,S)$ be an instance of the connected list triad-free cutset problem.
	First, create an instance $\mathcal I'=(U,V,S,\emptyset)$ of the list 2-clique cutset problem.
	If $\mathcal I'$ admits a solution $(U',V',S',\emptyset)$, then output $(U',V',S')$.
	Otherwise, for each of the induced paths $a\dd b\dd c$ in $S$ do the following.
	Create an instance $\mathcal I_{abc}=(U\setminus \{a,b,c\},V\setminus \{a,b,c\},S)$ of the list triad-free cutset problem.
	If $\mathcal I_{abc}$ admits a solution for some path $a\dd b \dd c$, then output it.
	Otherwise, output ``no solution exists''.
	
	\
	
	{\it Proof of correctness:}
	
	\
	
	Suppose that $\mathcal I$ admits a solution $(\tilde U, \tilde V,\tilde S)$.
	If $\tilde S$ is a clique, then $\mathcal I'$ admits a solution.
	It is also clear that every solution to $\mathcal I'$ is a solution to $\mathcal I$.
	
	Thus, we may assume that $\tilde S$ is not a clique.
	Since $\tilde S$ is connected, there exists an induced path $a\dd b \dd c$ in $\tilde S$.
	Therefore, $\mathcal I_{abc}$ admits a solution.
	For the converse, let $(U',V',S')$ be a solution to $\mathcal I_{abc}$.
	Observe that $\{a,b,c\} \subseteq S'$ and note that because $S'$ is triad-free, every vertex $s\in S'\setminus \{a,c\}$ is adjacent to either $a$ or $c$.
	Thus, $S'$ is connected, and therefore every solution to $\mathcal I_{abc}$ is also a solution to $\mathcal I$.
	This completes the proof.
	
\end{proof}

For the proof of Proposition \ref{pro:minimal-list-triad-free-cutset}, we will need the following definitions and Theorem \ref{thm:list-vertex-cover-np-hard}.
A {\em vertex cover} of a graph $G$ is a set of vertices $C$ such that every edge of $G$ has an endpoint in $C$.
A {\em minimal} vertex cover is a vertex cover $C$ such that no proper subset of $C$ is a vertex cover.
The {\em extension minimal bipartite vertex cover problem} takes as input a bipartite graph $G$ and a vertex set $X$ and asks to determine whether $G$ has a minimal vertex cover $C$ with $X\subseteq C$.

\begin{theorem}[Casel, Fernau, Ghadikolaei, Monnot, Sikora \cite{list-vertex-cover-np-hard}]
\label{thm:list-vertex-cover-np-hard}

	The extension minimal bipartite vertex cover problem is NP-hard.
	
\end{theorem}

\minimallisttriadfree*

\begin{proof}
	
	We show a reduction from the extension minimal bipartite vertex cover problem to the minimal list triad-free cutset problem.
	Let $G$ and $X\subseteq V(G)$ be an instance of the extension minimal bipartite vertex cover problem.
	Let $(A,B)$ be a bipartition of $G$.
	Obtain the graph $G'$ from $G$ by making $A$ and $B$ cliques and adding the vertices $a$ and $b$ such that $N(a)=A$ and $N(b)=B$.
	Note that $G'$ is the complement of a bipartite graph, so in particular $V(G)$ is bisimplicial.
	Let $A'=A\cup\{a\}$ and $B'=B\cup \{b\}$.
	Create an instance $\mathcal I=(A'\setminus X,B'\setminus X,A\cup B)$ of the minimal list triad-free cutset problem.
	
	Let $C$ be a minimal vertex cover of $G$ with $X\subseteq C$.
	We claim that $\mathcal S=(A'\setminus C,B'\setminus C,C)$ is a solution to $\mathcal I$.
	Let $\mathcal I'=\mathcal I$ be an instance of the list triad-free cutset problem.
	It is straightforward to check that $\mathcal S$ is a solution to $\mathcal I'$.
	To show the minimality of $\mathcal S$, let $c\in C$ be a vertex and suppose that $c\in A$.
	Then, $c$ is adjacent to $a\in A'\setminus C$ by construction.
	Since $C$ is a minimal vertex cover of $G$, it also follows that there exists a vertex $b'\in B\setminus C$ such that $c$ is adjacent to it.
	The case of $c\in B$ follows analogously.
	
	For the converse, let $(A'\setminus S,B'\setminus S,S)$ be a solution to $\mathcal I$.
	It is clear that $X\subseteq S$ and that $S$ is a vertex cover of $G$.
	Since $N(a)=A$ and $N(b)=B$, it follows that for every $s\in S$ there exists a vertex in $(A\cup B)\setminus S$ that is adjacent to $s$.
	Thus, $S$ is a minimal vertex cover of $G$.
	This completes the proof.
	
\end{proof}

\section{Acknowledgments}

We thank Malory Marin and Nicolas Trotignon for their involvement in the early stages of the work on Theorem \ref{thm:triad-free-cutset}.

\section{AI Disclosure}

The algorithm from Theorem \ref{thm:triad-free-cutset} was developed in a conversation with GPT-6 Astra.
The write-up is entirely due to the authors.

\bibliographystyle{plain}
\bibliography{ref}

\end{document}